\documentclass[11pt]{amsart}
\usepackage[T1]{fontenc}
\usepackage{amsmath, amsfonts, amsthm, amssymb}

\allowdisplaybreaks[3]
\usepackage{graphicx}

\usepackage{comment}

\usepackage[margin=1in]{geometry}
\usepackage{color}
\usepackage{amssymb}
\usepackage{amsmath}
\usepackage{amsthm}
\usepackage{url}
\usepackage{framed}
\usepackage{enumitem}
\usepackage{multirow}
\usepackage{setspace}
 \theoremstyle{definition}
 \newtheorem{defn}{Definition}
 \newtheorem{theorem}{Theorem}
 \newtheorem{lemma}{Lemma} 
  
 \newtheorem{corollary}{Corollary}
 \newtheorem{proposition}{Proposition}
 \newtheorem*{conjecture}{Conjecture}
 
 \newtheorem*{remark}{Remark}
 
\newcommand\ip[2]{\langle #1, #2 \rangle} 
\newcommand{\Z}{\mathbb{Z}} 
\newcommand{\C}{\mathbb{C}}

\newcommand{\Fr}{\mathcal{G}}
\newcommand{\eps}{\varepsilon}
\newcommand{\supp}{\text{supp}}
\newcommand{\br}[1]{\langle #1 \rangle}

\newcommand{\sgn}{\text{sgn}}
\newcommand{\rank}{\text{rank}}
\newcommand{\tensor}{\otimes}

\DeclareMathOperator{\Aut}{\mathrm{Aut}}

\newcommand{\be}{\begin{equation}}
\newcommand{\ee}{\end{equation}}

\begin{document}
\title{Towards a classification of incomplete Gabor POVMs in $\C^d$}

\date{\today}

\author{Assaf Goldberger}
\address{School of Mathematical Sciences, Tel-Aviv University, Tel-Aviv, Isreal, 69978} \email{assafg@tauex.tau.ac.il}
\author{Shujie Kang}
\address{Department of Mathematics,
University of Texas at Arlington, Arlington TX 76019, USA}
\email{shujie.kang@uta.edu}
\author{Kasso A.~Okoudjou}
\address{Department of Mathematics,
Tufts University, Medford MA 02131, USA} \email{Kasso.Okoudjou@tufts.edu}

\begin{abstract}
Every (full) finite Gabor system generated by a unit-norm  vector $g\in \C^d$ is a finite unit-norm tight frame (FUNTF), and can thus be associated with a (Gabor) positive operator valued measure (POVM). Such a POVM is informationally complete if the $d^2$ corresponding rank one matrices form a basis for the space of $d\times d$ matrices. A sufficient condition for this to happen is that the POVM is symmetric, which is equivalent to the fact that the associated Gabor frame is an equiangular tight frame (ETF). The existence of Gabor ETF is an important special case  of the Zauner conjecture. It is known that generically all Gabor FUNTFs lead to informationally complete POVMs.   In this paper, we initiate a classification of non-complete Gabor POVMs. In the process we establish some seemingly simple facts about the eigenvalues of the Gram matrix of the rank one matrices generated by a finite Gabor frame. We also use these results to construct some sets of $d^2$ unit vectors in $\C^d$ with a relatively smaller number of distinct inner products.  
\end{abstract}


\maketitle

\tableofcontents

\section{Introduction and background}\label{Sec:Intro}
Let $M$ and $T$ denote the modulation and translation operators defined  by 
$$\begin{cases}
Mg=(\omega^ng_n)_{n=0}^{d-1}\\
Tg=(g_{n-1})_{n=0}^{d-1}
\end{cases}$$ for $g=(g_n)_{n=0}^{d-1}\in \C^d$, and where $\omega = e^{2\pi i/d}$ is a $d^{th}$ root of unity.

The (finite) Gabor system $ \mathcal{G}(g):=\mathcal{G}(g, \Z_d\times \Z_d)=\{M^kT^\ell g\}_{k, \ell =0}^{d-1}$ generated by a unit vector $g\in \C^d$ is a Finite Unit Norm Tight Frame (FUNTF), that is 
$$\sum_{k, \ell =0}^{d-1}|\ip{x}{M^kT^\ell g}|^2= d^3\|x\|^2 \iff x=\tfrac{1}{d^{3}} \sum_{k, \ell=0}^{d-1}\ip{x}{M^k T^\ell g}M^k T^\ell g$$ for each $x\in \C^d$, see \cite{benedetto2003finite, waldron2018introduction}. 
It follows that $\mathcal{G}(g)$ can be cononically associated with a \emph{positive operator valued measure (POVM)}.  By definition, a (Gabor) POVM is \emph{informationally complete (IC)} if the ($d^2$) rank-one matrices $\{M^k T^\ell gg^*T^{-\ell}M^{-k}\}_{k, \ell=0}^{d-1}$  form a basis for $\C^{d^{2}}$ viewed as the space of $d\times d$ matrices. 
If in addition,  $$|\ip{M^k T^\ell gg^*T^{-\ell}M^{-k}}{M^{k'} T^{\ell'} gg^*T^{-\ell'}M^{-k'}}_{HS}|=|\ip{M^kT^\ell  g}{M^{k'}T^{\ell'}g}|^2$$ is constant for $(k, \ell)\neq (k', \ell')$, then we say that the Gabor POVM is \emph{symmetric (S)}. Here and in the sequel, we use $\ip{A}{B}_{HS}=\text{trace}(AB^*)$ to denote the Hilbert-Schmidt inner product on $\C^{d^{2}}$. It is a simple fact to check that every symmetric Gabor POVM is also informationally complete, but the converse is not necessarily true. 
In the sequel,  a (Gabor) FUNTF $\mathcal{G}(g)$ will be called a  SIC-POVM if the corresponding (Gabor) POVM is  symmetric and informationally complete. We refer to \cite{appleby2005symmetric, Renes2004} for more on POVMs.

It was conjectured by Zauner~\cite{zauner1999quantum} that for each $d\geq 1$ there exists a unit vector $g \in \C^d$ such that the Gabor system $\mathcal{G}(g)$ is a SIC-POVM. This conjecture is usually stated in the following form: 

\begin{conjecture}\cite{zauner1999quantum}
In any $\mathbb{C}^d$, $d>2$, there exists Gabor SIC-POVM generated by a single unit norm vector $g$ under the orbit of Heisenberg group. In particular,  for any $k,\ell\in \mathbb{Z}/d\mathbb{Z}\setminus\{(0,0)\}$, $|\langle g,M^kT^\ell g \rangle|=\frac{1}{\sqrt{d+1}}$, 
\end{conjecture}	
The conjecture remains open, and the search for  the generating vector (or fiducial vector) has been resolved  in dimensions $2-16, 19, 24, 28, 35, 48, 124,$ and $323$. Moreover, numerical solutions in dimensions up to $67$ have also been established.  We refer to \cite{appleby2012galois, grassl2017fibonacci, zauner1999quantum} and the references therein for more on the Zauner conjecture. 
	
We note that the (full) Gabor frame $\mathcal{G}(g)$ always generates a POVM, so the Zauner conjecture asserts that one can always find a generator so that this POVM is also symmetric, and hence informationally complete. The question of completeness of the Gabor POVM had long been investigated and it is known (\cite{Healy93, ArPeSa04}) that 	$\mathcal{G}(g)$ is an informationally complete if and only if \begin{equation}\label{completnesCond}
\ip{g}{M^kT^\ell g}\neq 0  \, \, \forall \, \, (k, \ell) \in \Z_d\times \Z_d.
\end{equation} 
 It is worth observing that the completeness of Gabor POVMs is central in the theory of phase retrieval \cite{BaCaEd06, BCMN14, BoFl16}. In particular, in \cite{BoFl16} it was proved that the injectivity of the phase retrieval map generated by a Gabor system $\mathcal{G}(g)$ is equivalent to~\eqref{completnesCond}.  Furthermore, this condition is generic in the sense that~\eqref{completnesCond} holds for ``almost all'' unit-norm vectors $g$. 
 
 In this paper, we offer a new proof of the completeness of Gabor POVMs. Our proof is based on the spectral analysis of the $d^2\times d^2$ Gram matrix associated to the rank-one matrices  $\{M^k T^\ell gg^*T^{-\ell}M^{-k}\}_{k, \ell=0}^{d-1}$. In particular, we prove that the $d^2$ eigenvalues of this Gram matrix are $\lambda_{k, \ell}=d|\ip{g}{M^k T^\ell g}|^2$ where $k, \ell=0, 1, \hdots, d-1.$ To the best of our knowledge this simple fact has never been stated in this manner before, though one of the results in \cite[Theorem 2.3]{BoFl16} asserts that the spectrum of this Gram matrix lies in the interval $[\min_{k, \ell} d|\ip{g}{M^k T^\ell g}|^2, \max_{k, \ell}d|\ip{g}{M^k T^\ell g}|^2].$ 
  After re-deriving~\eqref{completnesCond}, we (re) prove that this condition is generic in the sense that the set of all generators of such systems is open dense in the unit sphere of $\C^d$.  
 We then focus on classifying non-IC Gabor POVMs.  This is an interesting problem in its own right and is the primary motivation for this paper. We use some of our classification results  to construct $k-$ distance sets of $d^2$ elements in $\C^d$, where $k$ is proportional to $d$.

The rest of the paper is organized as follows. In Section~\ref{Sec:SpectrumG} we give two different characterization of Gabor IC-POVMs. Both characterizations are based on the spectral analysis of the Gramian of the rank-one matrices associated to the Gabor system. The first approach which was mentioned above is based on explicit formulaes for the eigenvalues of the Gram matrix. The second approach is achieved by formulating the problem using an algebraic framework that we hope to use to investigate other POVMs.  The algebraic approach also allows us to investigate the rank of the span of the rank-one matrices associated to the non IC Gabor POVMs. In Section~\ref{sect:rankinvariance} we discuss some invariance of the rank of these POVMs. We refer to \cite{IvJaMi20} for a related work dealing with POVMs generated by rank $2$ positive semidefinite matrices. 

\section{Informationally Complete POVMs}\label{Sec:SpectrumG}

The goal of this section is to re-derive the characterization of IC  Gabor POVMs. In the sequel, for every vector $g$ we define the \emph{support} of $g$ to be the set $\supp(g)=\{i| \ g_i\neq 0\}$. In addition, the cardinality of $\supp(g)$ will be denoted by $\|g\|_0.$
Given a unit vector $g\in \C^d$, the collection $$\Fr(g) = \{ g_{k, \ell}:=M^k T^{\ell} g \}_{(k,\ell)\in \mathbb{Z}_d\times \mathbb{Z}_d}$$ will denote the  (Gabor) FUNTF generated by $g$. We will view $\Fr(g)$ as an ordered multiset, with respect to the lexicographical order of $\Z_d\times\Z_d$, and we shall abuse notation and denote again by $\Fr(g)$  the $d \times d^2$ matrix whose columns are the vectors of this frame given in this order. 
To each vector $g_{k,\ell}$ we associate a rank-one matrix given by $$\Pi_{k,\ell} :=\Pi_{k, \ell}(g)=  g_{k,\ell} \otimes g_{k,\ell}= g_{k,\ell} g_{k,\ell}^{*}.$$

We seek a characterisation of all $g$ for which 
the rank-one matrices $\{\Pi_{k,\ell}\}_{k,\ell=0}^{d-1}$ span the space of all $d\times d$ matrices. This is the case if and only if the Gram matrix $$G(g):=(\ip{\Pi_{k,\ell}}{\Pi_{k',\ell'}})=(\text{tr}({\Pi_{k,\ell}^*\Pi_{k',\ell'}}))$$ of this set of matrices if full rank. This is also equivalent to the information completeness of the associated Gabor POVM.  Thus we focus on analyzing the rank of $G(g)$ by giving a full descriptions of its spectrum. The completness of Gabor POVM was also obtained in \cite[Theorem 15; Section IV-A]{Healy93} and in \cite[Section 4.1]{ArPeSa04}. Nonetheless, we point out that our result also gives the rank of the Gram matrix $G(g)$.

The main result of this section is the following theorem that gives the rank of the Gram matrix $G(g)$.

\begin{theorem}\label{thm:rankG}
Let $g\in \C^d$ be a unit vector. The rank of the Gramian $G(g)$ equals the number of nonzero entries in the multiset $\{\ip{g}{M^kT^{\ell}g}\}_{(k,\ell)\in \Z_d\times \Z_d}$. Consequently, the Gabor POVM is informationally complete if and only if $\ip{g}{M^kT^{\ell}g}\neq 0$ for all $k, \ell.$
\end{theorem}	

In Section~\ref{subsec:Gram} we compute all the eigenvalues of $G(g)$ obtaining its rank as a consequence. We note that \cite[Theorem 2.3]{BoFl16} shows that~\eqref{completnesCond} implies that $G(g)$ has full rank. Their proof consists of showing that~\eqref{completnesCond} implies that the lowest eigenvalue of $G(g)$ is positive. Our result shows that~\eqref{completnesCond} is also a necessary for $G(g)$ to be full rank. In Section~\ref{subsec:algebraic}, we use an  algebraic framework to give a second proof of Theorem~\ref{thm:rankG}. We then use this framework in Section~\ref{sect:rankinvariance} to identify some invariants of the rank of $G(g)$.

\subsection{Spectrum of the Gramian $G(g)$}\label{subsec:Gram}

In this section, we prove that the Gram matrix $G(g)$ is block circulant with circulant blocks. This structure allows us to compute the eigenvalues of $G(g)$ since it is diagonalizable by an appropriate DFT matrix. In particular, the  first part of result is an extension of a well-known fact about block circulant matrices \cite{tee2007eigenvectors, combescure2009block}. To the best of our knowledge, the remaining parts are new, though they seem quite simple. Let $F:=F_d$ denote the $d\times d$ DFT matrix. In particular, the $(k, \ell)$ entry of $F$ is $F_{k, \ell}=\omega^{k \ell }/\sqrt{d}$, where $\omega=e^{2\pi i /d}$.

\begin{theorem}\label{thm:Geigenvalue} Let $g$ be a unit vector in $\C^d$, and $G(g)=(\ip{\Pi_{k,\ell}}{\Pi_{k',\ell'}})$ be the Gram matrix of the rank-one projectors $\{\Pi_{k,\ell}=M^kT^\ell g\otimes M^kT^\ell g\}_{k, \ell=0}^{d-1}$. The following statements hold.
\begin{enumerate}
    \item[(i)] $G(g)$ is a block circulant matrix with circulant blocks. In particular, the $((k,\ell),(k',\ell'))$-th entry of $G$ is  
	\begin{equation*}\label{eqn:entryG}
		G_{((k,\ell),(k',\ell'))} \equiv \ip{\Pi_{k,\ell}}{\Pi_{k',\ell'}} =|\ip{g}{M^{k'-k}T^{\ell'-\ell}g}|^2,
	\end{equation*}
	where $k,k', \ell, \ell' \in \{0,\cdots,d-1\}$.
    \item[(ii)] $G(g)$ can be diagonalized by $F\otimes F$. In particular,  the eigenvalues $\{\lambda_{a, b}\}_{a, b=0}^{d-1}$ of $G(g)$ are given by
	\begin{equation*}
		\lambda_{a,b}=\sum\limits_{k,l=0}^{d-1} \omega^{ak+bl}|\ip{g}{M^kT^lg}|^2=d\left|\sum_{k=0}^{d-1} g_kg_{k+a}^*\omega^{bk} \right|^2=d|\ip{g}{M^bT^ag}|^2
	\end{equation*} where $a, b \in \Z_d.$
	
 		In particular, up to a factor of $d$, the entries of $G(g)$ are the same as its eigenvalues.
    \end{enumerate}
	
\end{theorem}

\begin{proof} Let $g$ be a unit vector in $\C^d$.
\begin{enumerate}
    \item[(i)] $G_{((k,\ell),(k',\ell'))}(g) = 	|\ip{g_{k,\ell}}{g_{k',\ell'}}|^2 = |\ip{M^kT^\ell g}{M^{k'}T^{\ell'}g}|^2 = |\ip{g}{M^{k'-k}T^{\ell'-\ell}g}|^2.$

 To prove that $G(g)$ is block circulant matrix with circulant blocks, we decompose $G(g)$ into $d^2$ blocks, each of the size $d\times d$ as following:
	\[
	G(g) = 
	\begin{bmatrix}
		A_{0,0} & A_{0,1} & A_{0,2} & \cdots & A_{0,d-1}\\
		A_{1,0} & A_{1,1} & A_{1,2} & \cdots & A_{1,d-1}\\
		\vdots  & \vdots  & \vdots  & \vdots & \vdots\\ 
		A_{d-1,0} & A_{d-1,1} & A_{d-1,2} & \cdots & A_{d-1,d-1}
	\end{bmatrix}
	\]
	The $(\ell,\ell')$-th entry in block $A_{k,k'}$ is then $G_{((k,\ell),(k',\ell'))}(g)$.

	First we show that $G$ is block circulant, that is, $A_{k,k'}=A_{k+1,k'+1}$ for any $k,k'\in \Z_d$. For any $\ell,\ell'\in \Z_d$ the $(\ell,\ell')$-entry in $A_{k+1,k'+1}$ is 
	$$G_{((k+1,\ell),(k'+1,\ell'))}=|\ip{ M^{k+1}T^\ell g}{ M^{k'+1}T^{\ell'}g }|^2=|\ip{g}{M^{k'-k}T^{\ell'-\ell}g}|^2.$$
	But this is exactly  the $(
	\ell,\ell')$-entry in $A_{k,k'}$. This shows that $G(g)$ is a block circulant matrix. Furthermore we have $A_{k,k'}=A_{0,k'-k}$. For simplicity, we denote $A_k \equiv A_{0,k}$, and note that
	\[
	G(g) = 
	\begin{bmatrix}
		A_0     & A_1 & A_2 & \dots & A_{d-1} \\
		A_{d-1} & A_0 & A_1 & \dots & A_{d-2} \\
		\vdots   & \vdots & \vdots & \vdots & \vdots\\
		A_1     & A_2 & A_3 & \dots & A_0
	\end{bmatrix}.
 	\]
	
	We next prove that each block $A_k$ is also a circulant matrix. Without loss of generality, we focus on the first row of the blocks. For any $k,\ell,\ell'\in \Z_d$, the $(\ell+1,\ell'+1)$-th entry of $A_k$ is equal to the $(\ell,\ell')$-th entry of $A_k$, since
	\begin{equation*}
	|\ip{T^\ell g}{ M^{k}T^{\ell'}g }|^2 = |\ip{ g}{ M^{k}T^{\ell'-\ell}g}|^2.
	\end{equation*}
	This concludes the proof of the first part of the Theorem.

    \item[(ii)] For any $a\in \Z_d$, consider the functions $h_a:\C^d\rightarrow \C^{d^2}$ defined as following
$$h_a(v) = \begin{bmatrix} v\\  \rho_a v\\  \rho_a^2 v\\ \vdots\\ \rho_a^{d-1} v\end{bmatrix},$$
where $\rho_a = \omega^a$. We claim that for any eigenvector $w$ of $G$, $w\in Range(h_a)$ for some $a\in \Z_d$. 

Denote $H_a= A_0 + A_1\rho_a + A_2\rho_a^2+ \cdots + A_{d-1}\rho_a^{d-1}$. By part (i), each $H_i$ is a circulant matrix, thus can be diagonalized by the $d\times d$ DFT matrix $F$. Suppose $v$ is an eigenvector of $H_a$ and $H_av=\lambda v$. Then
\[
Gh_a(v)=
\begin{bmatrix}
	A_0     & A_1 & A_2 & \dots & A_{d-1} \\
		A_{d-1} & A_0 & A_1 & \dots & A_{d-2} \\
		\vdots   & \vdots & \vdots & \vdots & \vdots\\
		A_1     & A_2 & A_3 & \dots & A_0
\end{bmatrix}
\begin{bmatrix}
v \\ \rho_a v \\ \vdots \\ \rho_a^{d-1} v 
\end{bmatrix}
=\begin{bmatrix}
H_av \\ \rho_aH_iv \\ \cdots \\ \rho_a^{d-1}H_iv
\end{bmatrix} = \lambda h_a(v).
\]

So the columns of $F \otimes F$ are eigenvectors of $G(g)$. Since $F  \otimes F$ is invertible, its columns account for all $d^2$ eigenvectors of $G(g)$.

We now find the spectrum of $G(g)$ which consists of  the collections of eigenvalues of $\{H_a\}_{a=0}^{d-1}$. Denoting the $(0,n)$ in $H_a$ as $H_a^n$, $n\in \{0,1,\cdots ,d-1\}$. Then the $d$ eigenvalues of $H_a$ are given by	
	\begin{align*}
		\lambda_{a,b} &= H_a^0+\rho_b H_a^1 + \rho_b^2 H_a^2 + \cdots + \rho_b^{d-1} H_a^{d-1}\\
		&= \sum\limits_{\ell=0}^{d-1} \rho_b^\ell H_a^\ell
		= \sum\limits_{\ell=0}^{d-1} \rho_b^\ell (\sum\limits_{k=0}^{d-1} \rho_a^k |\ip{g}{M^kT^\ell g}|^2)
		= \sum\limits_{\ell=0}^{d-1} \sum\limits_{k=0}^{d-1} \rho_b^k \rho_a^\ell |\ip{g}{M^kT^\ell g}|^2\\
		&= \sum\limits_{k,\ell=0}^{d-1}  \omega^{ak+b\ell} |\ip{g}{M^kT^\ell g}|^2\\
	\end{align*}
	Writing, $|\ip{g}{M^kT^\ell g}|^2=\sum_{n, m=0}^{d-1} g_ng_m^*g_{m-\ell}g_{n-\ell}^*\omega^{k(n-m)}$, we see that 
 		$$\lambda_{a,b}=\sum_{n, m,k,\ell=0}^{d-1} g_ng_m^*g_{m-\ell}g_{n-\ell}^*\omega^{k(n-m+a)+b\ell}.$$ Changing variables $r=n-\ell$, we obtain
 		$$\lambda_{a,b}=\sum_{n,m,k,r=0}^{d-1} g_ng_m^*g_{r+m-n}g_{r}^*\omega^{k(a+n-m)}\omega^{b(n-r)}.$$
 		Summing first over $k$, we may replace $\sum_{k=0}^{d-1} \omega^{k(a+n-m)}$ by $d\delta_{0,a+n-m}$. Then substituting $m=n+a$ we obtain that
 		$$ \lambda_{a,b}=d\sum_{n,r=0}^{d-1} g_ng_{n+a}^*g_{r+a}g_r^*\omega^{b(n-r)}=d\left|\sum_{n=0}^{d-1} g_ng_{n+a}^*\omega^{bn} \right|^2,$$
 		which is (ii).
	
	\end{enumerate}
\end{proof}

We can now give a proof of Theorem~\ref{thm:rankG}.
\begin{proof}[Proof of Theorem~\ref{thm:rankG}]
 Theorem~\ref{thm:rankG} follows directly from Theorem \ref{thm:Geigenvalue}. 
 \end{proof}

 	\begin{remark}
	\begin{enumerate}
\item  If we encode $g$ by a polynomial $f_g(X)=\sum_{k=0}^{d-1}g_k X^k\in \C[X]/(X^d-1)$, then Theorem~\ref{thm:rankG} can be rephrased as:
	
		The rank of $G(g)$ is the total number of zero coefficients in the collection of polynomials
		$$\left\{f_g(\omega^{-\ell}X)f_g(X^{d-1})\mod (X^d-1) \ | \ \ell=0,1,\ldots,d-1  \right\}.$$

	\item When $d=2$, the condition for a vector $g=(ce^{i\theta_1},\sqrt{1-c^2}e^{i\theta_2})$  to generate an informationally complete POVM is $c^2 \notin \{0,1,\frac{1}{2}\}$.
	\item $\lambda_{a,b}=\lambda_{d-a,d-b}$.
	\item Denote the matrix $\Lambda$ with entries $\Lambda_{a,b}=\lambda_{a,b}$, the matrix $\Xi$ with entries $\Xi_{k,l}=|\ip{g}{M^kT^lg}|^2$. Then $\Lambda = F \Xi F^*$.
	
	\end{enumerate}
\end{remark}

The next result shows that Gabor informationally complete POVMs are generic in the sense that the set of all such POVMs is an open dense in the set of all FUNTFs of $d^2$ elements in $\C^d$. 

\begin{proposition}\label{prop:generic} 
For each $d\geq 2$ there exists a complete Gabor POVM in $\C^d$. Moreover, the set of all normalized vectors $g$ such that $\rank(G(g))=d^2$ is open dense in the unit sphere in $\C^d$.
\end{proposition}

\begin{proof}
Consider the unit sphere $S^{2d-1}\subset \C^d$ as an algebraic subvariety of the affine $2d$ space with coordinates $Re(g_i)$ and $Im(g_i)$ for $1\le i \le d$. The condition $\rank(G(g))=d^2$ is given by the polynomial conditions $\lambda_{k,\ell}(g)=d|\sum g_i\bar{g}_{i+k} \omega^{-i\ell}|^2\neq 0$, for all $(k,\ell)$. For each pair $(k,\ell)$ with $k\neq 0$ take $g=(1,0,\ldots,1,0,\ldots)/\sqrt{2}$, with $\supp(g)=\{1,1+k\}$, so we have $\lambda_{k,\ell}(g)=d/4\neq 0$. Likewise, for $k=0$ we take $g=(1,0,\ldots,0)$, and again $\lambda_{0,\ell}(g)\neq 0$. Thus for each $(k,\ell)$ the set $S_{k,\ell}:=\{g\in S^{2d-1} \ | \ \lambda_{k,\ell}(g)\neq 0\}$ is nonempty open dense in the Zariski topology, and so is $S=\bigcap_{k,\ell} S_{k,\ell}$. \cite{ACM12, NS11}
\end{proof}

Clearly, if the support of $g$ is the set $\{1,2,\ldots,s\}$ for $s\le d/2$, then $\rank (G(g))<d^2$, regardless of what is $g$, as we have $g_i\bar{g}_{i+s+1}=0$ for all $i$. However, we can refine Proposition \ref{prop:generic} for supports of cardinality $>d/2$.

\begin{proposition}\label{prop:generic_in_support}
    Let $S\subseteq \{1,\ldots,d\}$ be a subset of cardinality $>d/2$. Then there exists a unit vectors $g$ with $\supp(g)=S$ and $\rank (G(g))=d^2$.
\end{proposition}

\begin{proof}
    Like in the proof of Proposition \ref{prop:generic} we consider the algebraic subvariety $V_S\subseteq S^{2d-1}$ of all vectors $g$ with $\supp(g)=S$. It in enough to find for each pair $(k,\ell)$ a vector $g\in V_S$ such that $\lambda_{k,\ell}(g)\neq 0$. For each $k$, the intersection $T:=S\cap (k+S)\neq \emptyset$, where we define $k+S=\{k+i\ | \ i \in S\}$. Pick up $j\in T$ and let $g'$ be the vector such that $g'_i=0$ if $i\notin S$, $g'_i=1$ if $i\in S$ and $i\notin\{ j,k+j\}$, and $g'_i=d$ if $i\in \{j,k+j\}$. Let $g=g'/||g'||$. Then the vector $w_k$ defined by $(w_k)_i=g'_i\bar{g'}_{i+k}$ satisfies $(w_k)_j=d^2$, $(w_k)_i=0$ if $i\notin T$, $(w_k)_i=1$ if $i\in T$ and $i\notin\{ j,j+k,j-k\}$ and $(w_k)_i=d$ otherwise. Therefore $\lambda_{k,\ell}(g)=d|\sum_{i\in T} (w_k)_i\omega^{i\ell}|^2/||g'||^4\ge d|d^2-(\#T-1) \cdot d |^2/||g'||^4>0$.
\end{proof}

In the next result we collect a number of properties about the rank of $G(g)$. In particular, we give all possible values of this rank when  $g$ is a unit vector in $\C^d$ whose support has size at most $2$. In Theorem~\ref{thm:rankd} we will prove a more general  result than part (iii).

\begin{proposition}\label{Prop:proertiesrankG}
For a unit vector $g\in \C^d$, we denote by $\|g\|_0$ the number of nonzero entries in $g$.  The following statements hold.
\begin{enumerate}
    \item[(i)] $\text{rank}(G(g))\geq d.$
    \item[(ii)] If $d$ is odd, the rank of the Gramian $G(g)$ is also an odd number.
    \item[(iii)] Suppose that $d\geq 2$ and   $\|g\|_0=1$. Then   $\rank(G(g))=d$.
    \item[(iv)] Suppose that $\|g\|_0=2$. Then   $\rank(G(g))=3d, 3d-1, 2d, \frac{3}{2}d$ or $d$. 
\end{enumerate}

\end{proposition}

\begin{proof} Let $g\in \C^d$ be a unit vector. For each $\ell \in \Z_d$, we let $w_\ell=(g_i\overline{g_{i+\ell}})_{i=0}^{d-1}\in \C^d$. 

\begin{enumerate}
    \item[(i)]  By part (ii) of Theorem~\ref{thm:Geigenvalue} we see that $\lambda_{0,0}=d$. In addition, note that since the diagonal entries of $G(g)$ are $|\ip{M^kT^\ell g}{M^kT^\ell g}|^2=1$ where $(k,\ell)\in Z_d\times Z_d$, we have $\text{tr}(G(g))=d^2$. If $\rank(G(g))<d$, then $\text{tr}(G(g))<d^2$. This is a contradiction. 
    \item[(ii)] Suppose that $d$ is odd. Recall that $\rank(G(g))$ equals to the number of nonzero eigenvalues of $G$, and that $\lambda_{0,0}=d>0$. When $(a,b)\neq (0,0)$, then  $(a,b)\neq (d-a,d-b)$ and $\lambda_{a,b}=\lambda_{d-a, d-b}$. Consequently,  $\rank(G(g))=1+2\sum_{(a,b)\in S}  \sgn(\lambda_{a,b})$, where $S=\{(a,b)\ | 0\le a\le d-1,\ 0\le b\le \frac{d-1}{2},\ (a,b)\neq (0,0)\}$.
     \item[(iii)]  Without loss of generality we assume that $g_0=1,$ and $g_k=0$ for all $k\neq 0$. Then $\|\hat{w}_0\|_0=d$ and $\|\hat{w}_\ell\|_0=0$ for all $\ell\neq 0$.
    \item[(iv)] Without loss of generality we assume that $g_0,g_\kappa\neq 0$ for some $\kappa\neq 0$.  

The proof is divided into two cases. 
\begin{enumerate}
    \item Suppose that $d$ is odd.
    $$\begin{cases}
\|w_0\|_0=2,\\
\|w_\kappa\|_0=\|w_{-\kappa}\|_0=1,\\
\|w_\ell\|_0=0 \text{ for } \ell\neq 0,\kappa,-\kappa.
\end{cases} \Longrightarrow \,\, \,  \begin{cases}
\|\hat{w_0}\|_0=\|\hat{w}_\kappa\|_0=\|\hat{w}_{-\kappa}\|_0=d,\\
\|\hat{w}_\ell\|_0=0 \text{ for } \ell\neq 0,\kappa,-\kappa.
\end{cases}$$

It follows that, $\rank(G(g))=3d$.
    \item Suppose now that $d$ is even.
    If $\kappa\neq \frac{d}{2}$ and $|g_0|\neq |g_\kappa|$, then by the same arguments we have $\rank(G(g))=3d$.
    
    Suppose $\kappa\neq \frac{d}{2}$ and $|g_0|= |g_\kappa|$. Then $\|\hat{w}_0\|_0=d-1$ when there exists an integer $c$ such that $\omega^{c\kappa}=-1$, otherwise $\|\hat{w}_0\|_0=d$. So $\rank(G(g))=3d-1$ or $3d$.

    Next, suppose that $\kappa=\frac{d}{2}$, then $w_\kappa=w_{-\kappa}$. It follows that 
    $$
    \begin{cases}
        \|w_\kappa\|_0=\|w_0\|_0=2,\\
\|w_\ell\|_0=0 \text{ for } \ell\neq 0,\kappa.
    \end{cases}$$
    Since $\omega^\kappa= -1$, we conclude that 
    \[ \|\hat{w_0}\|_0=
    \begin{cases}
    \frac{d}{2} &|g_0|=|g_\kappa|\\
    d &|g_0|\neq |g_\kappa|
    \end{cases}
    \text{ and }
    \|\hat{w}_\kappa\|_0=\begin{cases} 
      d & g_0\overline{g_k}\neq \pm g_k\overline{g_0} \\
      \frac{d}{2} & g_0\overline{g_k}= \pm g_k\overline{g_0} \\
   \end{cases}
\]
  Consequently,   $\rank(G)=3d, 3d-1, 2d, \frac{3d}{2}$ or $d$.
\end{enumerate}
\end{enumerate}
\end{proof}

\subsection{Algebraic structure of Gabor POVMs}\label{subsec:algebraic}

	In this section we give another proof of Theorem \ref{thm:rankG} from an algebraic point of view. In particular, this approach allows us to list several actions that leave invariant the rank of $G(g)$.

	The Weyl-Heisenberg group is the subgroup $W\subseteq U_d(\C)$ generated by $M$ and $T$. It has order $d^3$ and every element of $W$ can be written uniquely as $\omega^aM^bT^c$ for integers $0\le a,b,c<d$. The elements of $W$ are monomial matrices. A \emph{monomial} matrix is a square matrix with a unique nonzero entry in each row and column, which is a phase. A matrix $X$ is monomial, if and only if it can be written (uniquely) as a product $DP$, where $D$ is diagonal with diagonal entries of modulus $1$, and $P$ is a permutation matrix.
	Gabor systems give rise to a monomial representation of $W$ on the vector space $\C[\Fr(g)]$, $g$ is a symbolic vector of indeterminates. Here $\C[S]$ denotes the vector space of formal complex linear combinations of the set $S$. Explicitly, $w=\omega^aM^bT^c$ acts on basis elements $g_{k,\ell}=M^kT^\ell g$ by
	$$ \omega^aM^bT^c g_{k,\ell}=\omega^{a-ck}g_{k+b,\ell+c}.$$
	We introduce a monomial matrix $\mathcal M(w)$ of size $d^2$ to denote this action. This matrix is indexed by $(t,s)\in \Z_d\oplus \Z_d$ and its $((t,s),(t',s'))$ entry is given by
	$$ \mathcal{M}(w)_{((t,s),(t',s'))}=\begin{cases}\omega^{a-ct'}\, \, \, \text{if}\, \, \, (t, s)=(t'+b, s'+c)\\
	0\, \, \, \text{else}.\end{cases}$$

	It is clear by the construction that $\mathcal M$ is a group homomorphism: $\mathcal M(ww')=\mathcal M(w)\mathcal M(w')$. We define another monomial representation, by $\mathcal M'(w):=|\mathcal M(w)|$, the entrywise absolute-value.\\

	Consider the Gram matrix $H(g):=\Fr(g)^* \Fr(g)$, for a unit vector $g$. Note that this is the Gram matrix of the Gabor frame $\Fr(g)$, and is different from the Gram matrix $G(g)$ of the rank-one matrices we have discussed thus far.  In particular, for every $d\geq2$, $H(g)$ is a $d^2\times d^2$ matrix, and its eigenvalues are only $0$ and $d$.
	Moreover,  $\tfrac{1}{\sqrt{d}}H(g)$ is a self adjoint idempotent of rank $d$. The matrix $H(g)$ is invariant under the monomial action $\mathcal M$:
	$$\mathcal M(w)H(g)\mathcal M(w)^*=H(g), \ \forall w\in W.$$ The collection 
	$$\mathfrak A=\mathfrak A(\mathcal M):= \{X\in M_{d^2}(\C)\ | \ \forall w\in W, \ \mathcal M(w)X\mathcal M(w)^*=X\}$$ is a matrix algebra, closed under conjugate transpose. Similarly we have the matrix algebra
	 $$\mathfrak A'=\mathfrak A(\mathcal M')=\{X\in M_{d^2}(\C)\ | \ \forall w\in W, \ \mathcal M'(w)X\mathcal M'(w)^*=X\}.$$
	 Both algebras have dimension $d^2$.  In the following theorem, if $Z$ is a group, let $\C[Z]$ denote the group algebra over $\C$.  In addition, we shall make use of the $\C$-algebra isomorphism
	 \be\label{theta} \theta: \frac{\C[X]}{X^d-1}\to \C^d,\ee given by $X\mapsto (\omega^i)_{0\le i<d}$.
	 
	 \begin{theorem}\label{thm:1}
	 	There is an isomorphism of algebras over $\C$,
	 	\be
	 		\label{isom1} \eps: \mathfrak A \xrightarrow{\cong} M_d(\C).\ee
	 		There is a sequence of isomorphisms of algebras over $\C$
	 		\begin{multline} \label{isom2} \mathfrak A'  \xrightarrow{\cong} \C[\Z_d\oplus \Z_d]\xrightarrow{\cong} \C[\Z_d]\tensor_\C \C[\Z_d] \xrightarrow{\cong} \C[X]/(X^d-1)\tensor_\C \C[X]/(X^d-1)\\ \xrightarrow{\theta\tensor \theta\ \  \cong} \C^d \tensor_\C \C^d\cong \C^{d^2}.	 		
	 	\end{multline}
 	 The first isomorphism respects the conjugate-transpose.
 	 Let $\mu:\mathfrak A'\xrightarrow{\cong} \C^{d^2}$ denote the composition. Under $\mu$ the conjugate transpose becomes complex conjugation on $\C^{d^2}$.
	 \end{theorem}
 
     \begin{proof}
     	Let us construct first the inverse map $\eps^{-1}:M_d(\C)\to \mathfrak{A}$. The complex vector space $M_d(\C)$ has the special basis $\{M^iT^j\}_{0\le i,j<d}$. It is enough to define $\eps^{-1}$ on this basis, and show that it is an algebra homomorphism. For each $(i,j)$, we first construct a matrix $E(i,j)\in \mathfrak A$, which satisfies $E(i,j)_{(p,q),(0,0)}=\delta_{(i,j),(p,q)}$. 
     	The matrix $E(i,j)$ is uniquely determined by these properties. Indeed, by the invariance under the monomial matrices $\mathcal M(w)$, 
     	\begin{multline*}
     		E(i,j)_{(a+i,b+j),(a,b)}=\mathcal M(M^aT^b)_{(a+i,b+j),(i,j)} E(i,j)_{(i,j),(0,0)}\mathcal M(T^{-b}M^{-a})_{(0,0),(a,b)}\\ = \omega^{-bi} \cdot 1 \cdot \mathcal M(\omega^{-ab}M^{-a}T^{-b})_{(0,0),(a,b)}=\omega^{-bi}\cdot 1 \cdot \omega^{-ab+ab}=\omega^{-bi}.
     	\end{multline*} 
        A similar computation shows that $E(i,j)_{(a',b'),(a,b)}=0$ if $(a',b')\neq (a+i,b+j)$. This shows the uniqueness of $E(i,j)$. On the other hand, it is easy to check that if we define $E(i,j)$ by these formulae, then $E(i,j)\in \mathfrak A$.
        Note that $E(i,j)$ is monomial and that $|E(i,j)|=|\mathcal M(M^iT^j)|$.
        
        Next we compare $E(i+s,j+t)$ with $E(i,j)E(s,t)$. We have
        \begin{multline*}
        	(E(i,j)E(s,t))_{(a+i+s,b+j+t),(a,b)}=E(i,j)_{(a+i+s,b+j+t),(a+s,b+t)}E(s,t)_{(a+s,b+t),(a,b)}\\
        	=\omega^{-(b+t)i}\omega^{-bs}=\omega^{-b(i+s)}\omega^{-it}.
        \end{multline*}
        We also have 
        $$E(i+s,j+t)_{(a+i+s,b+j+t),(a,b)}=\omega^{-b(i+s)}.$$ Hence
        $E(i,j)E(s,t)=\omega^{-it}E(i+s,j+t)$. We try to define \be\label{eq:3} \eps^{-1}(M^iT^j)=\omega^{f(i,j)}E(i,j)\ \text{for some } f(i,j)\in \Z_d,\ee and we extend $\eps^{-1}$ by linearity to a vector-space homomorphism.  We wish to find an $f$, such that $\eps^{-1}$ will be actually an algebra homomorphism. A necessary and sufficient condition for this is that $\eps^{-1}(M^iT^j\cdot M^sT^t)=\eps^{-1}(M^iT^j)\cdot\eps^{-1}(M^sT^t)$. Then using \eqref{eq:3} and equality $M^iT^j\cdot M^sT^t=\omega^{-js}M^{i+s,j+t}$, we must have
        $$f(i,j)f(s,t)=\omega^{-js-it}f(i+s,j+t).$$
        This condition is satisfied by the choice $f(i,j)=\omega^{ij}$, so
        \be \label{eq:4}\eps^{-1}(M^iT^j):=\omega^{ij}E(i,j)\ee extends to a $\C$-algebra homomorphism $M_n(\C)\to \mathfrak A$.
        
        For $(a,b)\neq (c,e)$ modulo $d$, $\eps^{-1}(M^aT^b)$ and $\eps^{-1}(M^cT^e)$ have disjoint supports, hence $\eps^{-1}$ is injective. By comparing dimensions we conclude that $\eps^{-1}$ is an algebra isomorphism. Since both $M^iT^j$ and $\eps^{-1}(M^iT^j)$ are monomial matrices, then their conjugate-transpose is equal to their inverse, and $\eps^{-1}$ being a ring isomorphism respects matrix inverses. Hence $\eps$ commutes with the conjugate-transpose. This completes the first part.\\
     	
     	The map $\alpha:\Z_d \oplus \Z_d \to GL_{d^2}(\C)$ given by $\alpha(s,t)=\mathcal M'(M^aT^b)$ is a group homomorphism. Then we can extend $\alpha$ to a $\C$-algebra map $\eps':\C[\Z_d\oplus\Z_d] \to M_{d^2}(\C)$. This map is injective since as above $\{\mathcal M'(M^aT^b)\}_{(a,b)}$ have disjoint supports. Also, as $ww'$ and $w'w$ differ only by a phase, $\mathcal M'(w'w)=\mathcal M'(ww')=\mathcal M'(w)\mathcal M'(w')$, and in particular $\mathcal M'(w')$ is stable under conjugation with $\mathcal M'(w)$. It follows that the image of $\mathcal \eps'$ is in $\mathfrak A'$. Comparing dimensions, we obtain the first isomorphism in \eqref{isom2}. The other isomorphisms are well known and natural. They are given by the maps $[(a,b)]\mapsto [a]\tensor [b] \mapsto X^a \tensor X^b \mapsto [1,\omega^a,\omega^{2a},\ldots \omega^{(d-1)a}]\tensor [1,\omega^b,\omega^{2b},\ldots \omega^{(d-1)b}]\mapsto (\omega^{ia}\omega^{jb})_{i,j}$. Under those maps, the conjugate-transpose in $\mathfrak A'$ is compatible with the negation map $[(a,b)]\mapsto [(-a,-b)]$, which translates into complex conjugation at the rightmost term.

     	\end{proof}

     We remark that the isomorphism \eqref{theta}  $\theta:\C[X]/(X^d-1)\to \C^d$ is given by the Chinese Remainder Theorem for polynomials. That is, since $X^d-1=\prod_i(X-\omega^i)$, then
     $$\C[X]/(X^d-1)\cong \bigoplus_i \C[X]/(X-\omega^i)\cong \C^d.$$ The map $\theta$ takes a polynomial $f$ to the vector $(f(\omega^i))_i$. Then this linear operation is nothing but the discrete Fourier transform. If we write $f(X)=\sum_{i=0}^{d-1}f_iX^i\in \C[X]/(X^d-1)$ we have  $\theta(f)=\sqrt{d}F\cdot [f_0,f_1,\ldots,f_{d-1}]^T\in \C^d.$ Given an element $X\in \mathfrak A$, the matrix $X^{(2)}$ defined by $X^{(2)}_{i,j}:=|X_{i,j}|^2$ is a member of $\mathfrak A'$. In the spacial case $X=H(g)$, writing $\Fr(g)=\{g_{0,0},\ldots,g_{d,d}\}$ and letting $\Pi_{i,j}=g_{i,j}g_{i,j}^*$, then $G(g):=X^{(2)}_{i,j}=Tr(\Pi_{i,j}\Pi_{i,j}^*)\in \mathfrak A'$, and $rank \ G(g)$ is the dimension of the complex vector space spanned by the $\Pi_{i,j}$. We have
     
     \begin{lemma}\label{lem:ranklem}
     	Under the isomorphism \eqref{isom2} $\mu:\mathfrak A' \to \C^{d^2}$, the rank of $X\in \mathfrak A'$ equals the number of nonzero entries in $\mu(X)$.
     \end{lemma}
 
     \begin{proof}
     	By the isomorphism $\mu$, the algebra $\mathfrak A'$ contains a list of $d^2$ nontrivial idempotents $e_1,\ldots, e_{n^2}$, summing up to $1$ and satisfying $e_ie_j=0$ for $i\neq j$. On letting $\mathfrak A'$ act on a the vector space $V=\C^{d^2}$, we have $V=\bigoplus_j e_jV$, and this decomposition respects the action of $\mathfrak A'$. As $\mathfrak A'$ acts faithfully on $V$, and $e_k\mathfrak A'$ acts as zero on $e_jV$ for $j\neq k$, then $e_jV\neq 0$ for all $j$, and by equating dimensions we conclude that each $e_jV$ is 1-dimensional. Hence $\mu$ is nothing but a simultaneous diagonalization of the algebra $\mathfrak A'$ and the lemma follows.
     \end{proof}

\section{Invariance properties of the rank of Gabor POVMs}\label{sect:rankinvariance} 
In this section we identify a number of transformations leaving the rank of $G(g)$ invariant, and moreover leave invariant the multiset of the internal angles between the vectors of $\mathcal G(g)$. We exploit this to first classify all unit-norm vectors $g$
for which the rank of $G(g)$ is $d$. Next, we introduce a notion of an automorphism group of Gabor frames, and construct examples of $m$-distance sets with small $m$. Finally, we prove that $rank(G(g))$ cannot take a value strictly between $d$ and $2d$ when $d>2$ is a prime.\\

For a nonzero vector $g=(g_0, g_1, \hdots, g_{d-1})^{T}\in \C^d$, we let  
 \[a_{k,\ell}=a_{k,\ell}(g):=\sum_n g_n\bar g_{n+k}\omega^{n\ell}.\]

We recall that the $k$-translation $k+S$ of a subset $S\subseteq \Z_d$, is the set $\{k+s| \ s\in S\}$. We shall make the convention that if $r$ is a real number, then $\omega^r:=e^{2\sqrt{-1}r\pi/d}$. For $m\in \Z_d^\times$, let $m^{-1}\in \Z_d^\times$ denote its group inverse. Given $g\in \C^d$, it is not difficult to prove that each of the following transformations of $g$ resulting in a vector $h$ preserves the ranks \(\rank(G(h))=\rank(G(g))\) and the following transformation rules:

\begin{enumerate} \label{actions}
	\def\labelenumi{\arabic{enumi}.}
	\item \textbf{\emph{Phase}}: Let $h=cg$ for $|c|=1$. Then $a_{k,\ell}(h)=a_{k,\ell}(g)$.
	\item
	\textbf{\emph{Additive translation:}} Let \(h_i=g_{i+t}\).
	We have \be a_{k,\ell}(h)=\omega^{-t\ell}a_{k\ell}(g).\ee
	\item
	\textbf{\emph{Multiplicative}}: Let \(h_i=g_{mi}\), where 
	\(gcd(m,d)=1\). We have
	\be\label{mult1}  a_{k,\ell}(h)=a_{mk,m^{-1}\ell}(g).\ee
	\item
	\textbf{\emph{(Phase) Quadratic:}} Let
	\(h_i=g_i\omega^{a\binom{i}{2}+bi+c}\), where$c$ is a real number, $a$ is an integer,  and $b$ is an integer of a half integer such that 
	$a(d-1)/2+b$ is an integer. We have 
	\be\label{phase1} a_{k,\ell}(h)=a_{k,\ell-ak}(g)\omega^{-bk-a\binom{k}{2}}.\ee 
	We remark that the condition on $b$ implies that the transformation is well-defined if we consider the index $i$ as real integer. 
	\item[4'.] \textbf{\emph{(Phase) Quadratic:}} Suppose that  $\supp(g)\subset \kappa \Z_d$, for some positive integer $\kappa | d$, and let
	\(h_i=g_i\omega^{\tfrac{a}{\kappa}\binom{i}{2}+bi+c}\) for $i\in supp(g)$ and $h_i=0$ otherwise, where $c$ is a real number, $a$ is an integer and $b$ is an integer of a half integer such that $a(d-1)/2+b$ is  an integer. We have for every integer $s$
	\be a_{s\kappa,\ell}(h)=a_{s\kappa,l-as}(g)\omega^{-bs\kappa-\tfrac{a}{\kappa}\binom{s\kappa}{2}}.\ee
\end{enumerate}

\subsection{Characterization of rank $d$ Gabor POVMs}\label{subsec:rankd}
In this section give a complete characterization of all vectors $g$ leading to rank $d$ Gabor POVMs. We note that this generalizes part (iii) of Proposition~\ref{Prop:proertiesrankG}.

\begin{theorem}\label{thm:rankd}
Let $r$ be a divisor of $d$, and define a vector 
$g:=g(r,d)\in \C^d$  by
$$g(r,d)_i=\begin{cases}
   1 & i\equiv 0 \mod r\\
   0 & \text{ otherwise}.
\end{cases}$$
Then the rank of $G(g)$ is $d$, and $\mathcal{G}(g)$ is a 2-distance set. Conversely,  up to translation and phase quadratic transformations, a normalized vector $g$ has rank $d$, if and only if $g=\tfrac{1}{\sqrt{d/\kappa}}g(\kappa,d)$ for some $\kappa | d$.
\end{theorem}

\begin{proof}
For $g$ defined as above, we see that $a_{k,\ell}(g)=\sum_n g_n\bar g_{n+k}\omega^{n\ell}=0$ if $k$ is not a multiple of $r$. Otherwise $a_{tr,\ell}=\sum_{n=rj}g_{rj}\bar g_{r(j+t)}\omega^{n\ell}=\sum_{n=rj}\omega^{n\ell}$. This quantity vanishes if and only if $\omega^{r\ell}=1$. Equivalently $\ell$ is a multiple of $d/r$. Hence $a_{tr,sd/r}(g)$ are precisely the ones that do not vanish and there are $d$ of them.

	It remains to prove the only if part. Assume that $\rank(G(g))=d$ for a unit vector $g$.
	One eigenvalue of $G(g)$ is $d|a_{0,0}(g)|^2=d\left|\sum_i |g_i|^2\right|=d$.
	Notice that $|a_{k,\ell}(g)|\le \sum_i |g_i||g_{i+k}|\le \|g\|^2 =1$ by the Cauchy-Schwartz inequality, hence $d$ is the largest eigenvalue. 
	As $\text{tr}(G(g))=d^2$ and the rank is $d$, we must conclude that $d$ appears with multiplicity $d$, and all other eigenvalues are $0$.

	Let $S=\supp(g)$, and let $\kappa$ be the smallest positive integer such that $i,\kappa+i\in S$. Then the vector $(g_ig_{i+\kappa})_i\neq 0$, and $(a_{\kappa,\ell})_\ell$ is its discrete Fourier transform. It follows that $a_{\kappa,\ell}\neq 0$ for some $\ell$. If $\kappa=d$ then $S$ is a singleton and $g=g(d,d)$ up to phase and translation. So we shall assume from now that $\kappa<d$.
	
	Now $|a_{\kappa,\ell}(g)|=|a_{0,0}(g)|=1$ implies that 
	$$1=|\sum_{i\in S, i+\kappa\in S}g_i\bar{g}_{i+\kappa}\omega^{i\ell}|\leq \sum_{i\in S, i+\kappa\in S}|g_ig_{i+\kappa}|\leq \left(\sum_{i\in S,\ i+\kappa\in S} |g_i|^2 \right)^{1/2}\left(\sum_{i\in S,\ i+\kappa\in S} |g_{i+\kappa}|^2 \right)^{1/2},$$  where we applied the Cauchy-Schwartz inequality once more.
	This can happen if and only if $S\cap (\kappa+S)=S$, so $S$ must be $\kappa$-periodic. In particular $\kappa$ must divide $d$.
	
	By applying translation,  we may assume without loss of generality that $S$ is the subsets of multiples of $\kappa$, and for a given $k$ there exists an $\ell$ such that $a_{k,\ell}\neq 0$, if and only if $k=t\kappa$. Take $t=1$. By applying a phase quadratic transformation (type 4'), we may assume that $|a_{\kappa,0}|=1$. By Cauchy-Schwartz we conclude that the vectors $(g_i)$ and $g_{\kappa+i}$ are proportional, so $g_i=\alpha g_{\kappa+i}$ for all $i$, with $|\alpha|=1$. Modifying by a phase we may assume that $g_0=1/\sqrt{d/\kappa}$ hence $g_{\kappa j}=\alpha^j/\sqrt{d/\kappa}$. In particular for $j=d/\kappa$ we obtain $\alpha^{d/\kappa}=1$. Modifying by linear phase (type 4) we obtain $g=g(\kappa,d)/\sqrt{d/\kappa}$.
\end{proof}

\begin{remark} Choosing $r=d$, we recover part (iii) of Proposition~\ref{Prop:proertiesrankG}.
\end{remark}

\subsection{Examples of automorphisms of two types of Gabor systems}

The list of transformations 1--4' above gives rise to a notion of an automorphism group of a Gabor systems. We define the group $\mathcal G_d$ as the group of all maps $\C^d \to \C^d$ generated by the transformations 1--4'. For every unit vector $g$, let $\Aut(g)\subseteq \mathcal G_d$ be the subgroup fixing $g$. We call this the \emph{Automorphism} group of $g$. We can easily create vectors $g$  with nontrivial automorphisms. For example, if $g_{i+\delta}=g_i$ for all $i$ and $\delta | d$, then $g$ is a periodic vector having nontrivial translations as automorphisms. Likewise if $g_{mi}=g_i$ for all $i$, $\gcd(m,d)=1$, then $g$ has multiplicative automorphisms. A more interesting example will be a vector $g$, such that for all $i$, $g_{i+\delta}=g_i\omega^{\alpha\binom{i}{2}+\beta i+\gamma}$, for fixed $\alpha,\beta,\gamma,\delta\in \Z_d$. So any such $g$ has a nontrivial automorphism, which is a composition of phase-quadratic and a translation transformations. If $gcd(\delta,d)=1$, then by iterating this relation, it is easy to show that 
\begin{equation}\label{symm}
    g_i=\frac{\phi}{\sqrt{d}} \omega^{a\binom{i}{3}+b\binom{i}{2}+ci}, \ \ \ |\phi|=1,
\end{equation} 
where $a,b,c\in \Z_d$ solve uniquely the linear system
\begin{equation*}
\left\{\begin{array}{ll}
     &  a\delta=\alpha \\
     &  a\binom{\delta}{2}+b\delta=\beta\\
     & a\binom{\delta}{3}+b\binom{\delta}{2}+c\delta=\gamma.
\end{array}   \right.
\end{equation*}
    
Conversely, every vector of the form \eqref{symm} has this specific automorphism. A vector $g=(g_i)$ satisfying \eqref{symm} is what is known in the literature as an \emph{Alltop sequence}. Such sequences were constructed by Alltop \cite{1056185} for applications of spread spectrum radars and communication. 
By the transformation rule of a phase quadratic symmetries, we have that
$$ |a_{k,\ell}(g)|=|a_{k,\ell-\alpha k}(g)|,$$ which means that zero eigenvalues may be duplicated by the automorphism, giving some limitations on $\rank (G(g))$.

Assume now that $d$ is prime and $1\le a <d$. Then the above symmetry implies that $|a_{k,\ell}(g)|$ is independent of $\ell$. In fact we can show

\begin{proposition}(See \cite{6612625,inbook})
    If $1\le a <d$, $d$ is an odd prime, and $g$ is as in \eqref{symm}, then 
    $$|a_{k,\ell}(g)|^2=\begin{cases}
          \frac{1}{\sqrt{d}} & k>0\\
          0 & k=0, \ell>0
    \end{cases}. $$
    In particular $\rank (G(g))=d^2-d+1$. Moreover, the Gabor System $\{T^\ell M^kg\}_{\ell, k=0}^{d-1}$, together with the standard basis is a maximal MUB of $d+1$ bases.
\end{proposition}

\begin{proof}
   We compute directly for $k\neq 0$:
   \begin{multline*}
       |a_{k,\ell}(g)|=\left|\sum_i g_i \bar{g}_{i+k} \omega^{i\ell}\right| =\left|\frac{|\phi|^2}{d}\sum_i \omega^{a\binom{i}{3}+b\binom{i}{2}+ci}\omega^{-a\binom{i+k}{3}+b\binom{i+k}{2}+c(i+k)}\omega^{i\ell}\right|\\ =\frac{1}{d}\left|\sum_i \omega^{i\ell-ak\binom{i}{2}-ai\binom{k}{2}-bki}\omega^{-a\binom{k}{3}-b\binom{k}{2}-ck}\right|
       =\frac{1}{d}\left|\sum_i\omega^{-ak\binom{i}{2}-ai\binom{k}{2}-bki+i\ell}\right|\\
       =\frac{1}{d}\left| \sum_i \omega^{-\frac{ak}{2}\left(i-\frac{1}{2}-\frac{1}{ak}(\ell+bk+a\binom{k}{2})  \right)^2}\right|
       =\frac{1}{d}\left|\sum_i \omega^{-\frac{ak}{2}i^2}\right|=\frac{1}{d}|\pm \sqrt{\pm d}|=\frac{1}{\sqrt{d}}
   \end{multline*}
   In this computation we have interpreted a fraction $m/n$ as the unique integer $f$ such that $m\equiv nf \mod d$. We have used the fact that for an odd prime number $d$, and $gcd(\delta,d)=1$, $\sum_i \omega^{\alpha i^2}=\pm \sqrt{\pm d}$. This is a variant of the well-known Gauss sums \cite[Theorem 1]{6612625}.

   When $k=0$, $a_{k,\ell}(g)=\frac{1}{d}\sum_i \omega^{i\ell}=0$ if $\ell\neq 0$. This in particular means that $g$ is orthogonal to $M^ig$ for all $i$, and moreover  $B_0:=\{M^jg\}_j$ is an orthonormal basis. It follows that for every $i$, $B_i=\{T^iM^jg\}_i$ is an orthonormal basis, and that the Gabor system is a MUB of $d$ bases.
   Since all entries of $\mathcal G(g)$ are of the same modulus, we can adjoin the standard basis to obtain a maximal MUB.
\end{proof}

We next  construct another family of vectors $g$ having a nontrivial automorphism group, and leading to a non complete Gabor POVM with few distinct inner products.

We consider the subspace of $\C^d$ given by 
\begin{equation}\label{maut}
  \mathcal V(a,b,c,\kappa)=\{ g\in \C^d:  g_{\kappa i}=\omega^{a\binom{i}{2}+bi+c}g_i\}
    \end{equation}
for some integers $a,b,c\in \Z_d$ and $\kappa\in \Z_d^\times$.\\

The group $(\mathbb Z_d)^\times$ acts on the set $\Z_d$ by multiplication. The subgroup $\br{\kappa}$ generated by $\kappa$ yields a disjoint decomposition of $\br{\kappa}$- orbits
$$ \Z_d=\bigsqcup_r Q_r.$$ 
For each $r$ fix a point $i_r\in Q_r$. Then the value $g_{i_r}$ determines uniquely the values of $g_i$ for all $i\in Q_r$. Namely, if $i=\kappa^m i_r$, then by iterating the condition in \eqref{maut} we get
\be\label{oo} g_i=\omega^{\sum_{j=0}^{m-1} a\binom{\kappa^ji_r}{2}+b\kappa^ji_r+c}g_{i_r}.\ee
Notice however, that this must apply to any index $s$ such that $\kappa^si_r\equiv i_r\mod d$. Hence if $g_{i_r}\neq 0$,  the following condition must hold
\be\label{orient}
\sum_{j\in Q_r} a\binom{j}{2}+bj+c \equiv 0 \mod d.
\ee
In this case, we say that $Q_r$ is an \emph{orientable} $\br{\kappa}$-orbit for the triple $(a,b,c)$. Otherwise it is \emph{non-orientable}. We have 

\begin{proposition}\label{lem:2} Using the above notations and definitions, the following statements hold.  \begin{itemize}
    \item[(a)] The dimension over $\mathbb C$ of the space $\mathcal V(a,b,c,\kappa)$ is the number of orientable $\br{\kappa}$-orbits for the triple $(a,b,c)$.
    \item[(b)]If $d$ is odd and $gcd(d,\kappa^2-1)=1$, then all the $\br{\kappa}$-orbits are orientable for $(a, b, c)$.
\end{itemize}
    
\end{proposition}

\begin{proof} Let $\mathcal Q$ be the set of all $\br{\kappa}$-orbits in $\Z_d$, and denote by $N$ the number of orientable $\br{\kappa}$-orbits.
\begin{enumerate}
    \item[(a)]   Pick an index $i_r\in Q_r$ for any orbit $Q_r\in \mathcal Q$.  We define a linear map $E:\mathcal V(a,b,c,\kappa)\to \C^{\mathcal Q}$ by sending $g$ to the vector $(g_{i_r})_r$. The map $E$ is injective, because by \eqref{oo} $g$ is determined uniquely by the collection $\{g_{i_r}\}$. Also, when \eqref{orient} is not satisfied, then $g_{i_r}=0$, so we conclude that $\dim \mathcal V(a,b,c,\kappa)=\dim Image(E)\le N$. To prove the equality, for every orientable orbit $Q_r$ we will construct a vector $g\in \mathcal V(a,b,c, \kappa)$ with $g_{i_r}=1$ and $g_{i_s}=0$ for all $s\neq i$. We define $g_i$ for $i\in Q_r$ by equation \eqref{oo} taken with $g_{i_r}=1$, and for $i\notin Q_r$ we set $g_i=0$ for. Then condition \eqref{orient} guarantees that $g$ is well-defined, regardless of the choice of a lifting $m\in \Z$, and it is easy to check now that the condition in \eqref{maut} is satisfied. Hence $g\in \mathcal V(a,b,c, \kappa)$, and we are done proving (a).
\item[(b)] We must check condition \eqref{orient}. Since $d$ is odd and hence $2$ is invertible modulo $d$, it is enough to prove that for each orbit $Q_r$, $\sum_{j\in Q_r} j^2 \equiv \sum_{j\in Q_r} j \equiv 0 \mod d$. Since $Q_r=\{\kappa^t i_r\ | \  0\le t<s\}$, where $\kappa^si_r=i_r$, then using $\gcd(d,\kappa^2-1)=1$:
$$\sum_{t=0}^{s-1} (\kappa^ti_r)^2=\frac{\kappa^{2s}-1}{\kappa^2-1}i_r^2=\frac{i_r-i_r}{\kappa^2-1}i_r=0,$$
where the quantities are considered as elements of $\Z_d$. A similar argument proves that $\sum_{j\in Q_r} j\equiv 0\mod d$.
\end{enumerate}
\end{proof}

Let $g$ be a vector satisfying \eqref{maut}. By using \eqref{mult1} and \eqref{phase1} we have $|a_{\kappa i,\kappa^{-1}\ell}|=|a_{i,\ell-ai}|$ for all $(i,\ell)$, or equivalently
\be\label{trans} |a_{i,\ell}|=|a_{\kappa^{-1} i,\kappa\ell-\kappa^{-1} ai}|.\ee
Assume now the conditions of Proposition~\ref{lem:2}. 
Our next goal is to estimate the number distinct angles associated with the Gabor system generated by a vector $g$ that satisfies \eqref{maut}. Towards this we transform coordinates on $(\Z_d)^2$ by $i'=i$ and $\ell'=\ell-2a\kappa\ell/(\kappa^2-1).$ Write 
$a'_{i',\ell'}(g)=a'_{i',\ell'}=a_{i,\ell}(g)$. Then \eqref{trans} transforms to the simpler form
\be\label{symm1} |a'_{i,\ell}|=|a'_{\kappa^{-1}i,\kappa \ell}|.\ee
Recall that we also have the conjugacy symmetry $|a_{i,\ell}|=|a_{-i,-l}|$. Equivalently, 
\be\label{symm2} |a'_{i,\ell}|=|a'_{-i,-l}|.
\ee

Let $B_\kappa$ be the subgroup of $\Z_d^\times$ generated by $\kappa$ and $-1$. Then \eqref{symm1}-\eqref{symm2} are equivalent to
\be\label{symm3} |a'_{i,\ell}|=|a'_{t^{-1}i,t \ell}|, \ \ \text{for all } t\in B_\kappa.
\ee
The group $B_\kappa$ satisfies $B_\kappa=\br{\kappa}$ or $[B_\kappa:\br{\kappa}]=2$.

\begin{theorem}
Suppose that $d$ is odd and that $gcd(\kappa^2-1,d)=1$. Then the Gabor system $\mathcal G(g)$ is $m$-angular where 
\be\label{orbits} m\le \sum_{d_1,d_2|d}\frac{gcd(\delta(d_1),\delta(d_2))}{r}\varphi\left(\frac{d}{d_1}\right)\varphi\left(\frac{d}{d_2}\right).\ee
Here $r=|B_\kappa|$, $\delta(d_i)$ is the order of the image of $B_\kappa$ in $\Z_{d_i}^\times$, and $\varphi$ is the Euler totient function.

\end{theorem}

\begin{proof}
    The number of angles in this Gabor system is the number of distinct $|a_{i,\ell}(g)|$, which is the number of distinct $|a'_{i,\ell}(g)|$. The group $B_\kappa$ acts on $(\Z_d)^2$ via $(i,\ell)\mapsto (t^{-1}i,t\ell)$ for all $t\in B_\kappa$, hence by \eqref{symm3} $|a'_{i,\ell}|$ are constant along the orbits. We will be done if we show that the right hand side of \eqref{orbits} is the number of $B_\kappa$-orbits.
    
   We note that if $C$ is a finite group acting on a finite set $S$, and $\text{Fix}(s)$ is the cardinality of the stabilizer of the point $s\in S$, then the number of orbits is $\sum_{s\in S} 1/\text{Fix}(s)$.  Let  $\kappa_0$ be a generator for the cyclic group $B_\kappa$, and consider a point $(i,\ell)\in (\Z_d)^2$. Let $d_1=gcd(i,d)$ and $d_2=gcd(\ell,d)$. Then $\kappa_0$ has multiplicative order $\delta(d_1)$ modulo $d_1$ and $\delta(d_2)$ modulo $d_2$. Thus $\br{\kappa_0^{\delta(d_1)}}$ is the stabilizer of $i$, and $\br{\kappa_0^{\delta(d_2)}}$ is the stabilizer of $\ell$. It follows that $\text{Fix}((i,\ell))=r/gcd(\delta(d_1),\delta(d_2))$. The number of the pairs $(i,\ell)$ with $gcd(i,d)=d_1$ and $gcd(\ell,d)=d_2$ is $\varphi(d/d_1)\varphi(d/d_2)$. The theorem follows.
\end{proof}

\begin{corollary}\label{cor:1} Suppose that $d$ is prime and $\kappa$ generates $(\Z_d)^\times$, or that $d\equiv 3 \mod 4$ and $\kappa$ has order $(d-1)/2$.  Then, we have 
$$\begin{cases}
m\le d+2\\
\rank(G(g))\equiv d^2 \mod(d-1).
\end{cases}
$$
\end{corollary}

\begin{proof}
 In the both cases $B_\kappa=\Z_d^\times$ and $r=d-1$. The pair $(d_1,d_2)$ takes the values $(1,1),(d,1),(1,d),(d,d)$. We have $\delta(1)=1$ and $\delta(d)=d-1$. The right hand side of \eqref{orbits} is $d+2$.
 
 The second equation follows from the fact that all the non-trivial $B_\kappa$-orbits of $\Z_d^2$ are of order $d-1$.
\end{proof}

  We illustrate Corollary~\ref{cor:1} with an example, namely the family of Gabor frames generated by the Bj\"orck sequences.
  Suppose $d$ is an odd prime, and denote $\chi[k]\equiv \big( \frac{k}{d}\big)$ the Legndre symbol. We can define a vector in $\C^d$ accordingly as following:
\begin{itemize}
    \item When $d$ is prime and $d \equiv 1 \mod 4$, 
	\[g_k = \frac{1}{\sqrt{d}} e^{i\theta\chi[k]}, \textit{ where } \theta = \arccos\bigg(\frac{1}{1+\sqrt{d}}\bigg),\]
	for all $k\in \Z_d$.
	\item When $d$ is prime and $d \equiv 3 \mod 4$,
	\[
	g_k = 
	\begin{cases}
	\frac{1}{\sqrt{d}}e^{i\phi} & \textit{if } k \in \mathcal{Q}^C \subseteq (\Z_d)^{\times},\\
	\frac{1}{\sqrt{d}} & otherwise,
	\end{cases}
	\]	
	for all $k\in \Z_d$. Where $\phi=\arccos(\frac{1-d}{1+d})$, and $\mathcal{Q}^C$ is the preimage of $-1$ under $\chi$.
\end{itemize}

Bj\"orck sequences, which were constructed in \cite{MR899531} are CAZAC (Constant Amplitude Zero Auto Correlation) sequences, meaning that $\ip{g}{T^\ell g}=0$ for all $\ell \in \Z_d^\times$. For this reason, they provide examples of vectors $g\in \C^d$ with $\|g\|_0=d$ and do not generate IC-POVMs. In fact, the number of different values in the Gramian $G(g)$ is relatively small. When $d\equiv 3 \mod 4$, we know that $\rank(G(g))\leq d^2-2d+2$ by the result in \cite{MR2921081} and the following observation.

\begin{proposition}\label{Thm:NIBjorck}
Suppose $d$ is prime and $d\equiv 3 \mod 4$ and $g$ is a unit vector in $\C^d$, then $|\ip{g}{M^kT^\ell g}|$ take $d+1$ different values. Furthermore, $|\ip{g}{M^kT^\ell g}|=|\ip{g}{M^{k'}T^{\ell'}g}|$ if $k\ell\equiv k'\ell' \mod d$ and $(k,\ell),(k',\ell')\neq (0,0)$.
\end{proposition}

Note that by Corollary~\ref{cor:1}, the upper bound of the number of angles is $d+2$. Under the specific choice of of the phase $\phi$, the values of $a_{0,\ell}$ and $a_{\ell,0}$, $\ell\neq 0$ all degenerate to $0$. General choices of $\phi$ will result in $d+2$ angles.

\begin{remark}
We provide an example illustrating Proposition~\ref{Thm:NIBjorck} with the Bj\"orck sequence of length $d=7$ is
\[g_k=
\begin{cases}
\frac{3}{4\sqrt{7}}+i\frac{1}{4} &k=3,5,6.\\
\frac{1}{\sqrt{7}} & k=0,1,2,4.
\end{cases}
\]

Figure~\ref{fig:Bjorckp7} shows the different values of $|\ip{g}{M^kT^\ell g}|$.

\begin{figure}[h]
\includegraphics[width=8cm]{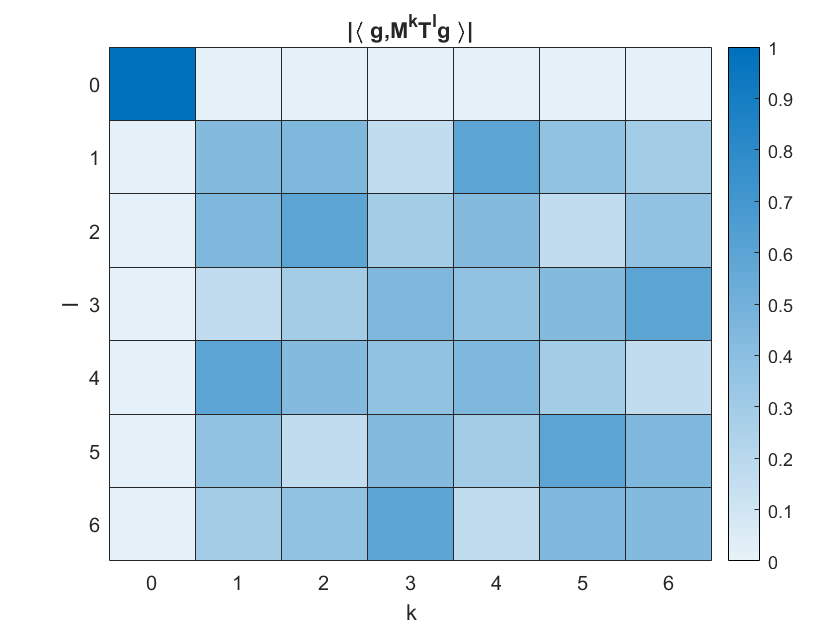}
\caption{Heatmap of $|\ip{g}{M^kT^\ell g}|$, where $g$ is the Bj\"orck sequence of length 7.}
\label{fig:Bjorckp7}
\end{figure}

\end{remark}



We conclude this section by showing that 
$\rank(G(g))$ cannot lie in the interval $(d, 2d)$ when $d$ is odd and prime. We contrast this with Proposition \ref{Prop:proertiesrankG} where we gave an example of $g$ leading to $\rank(G(g))= 3d/2$  for $d$ even.  But first, we need the following Lemma. 


\begin{lemma} \label{lem:supp}
	Assume that $d$ is an odd prime.
	If $\|g\|_0>d/2$ and $\rank(G(g))\neq d$, then $\rank(G(g))\ge 2d$.
\end{lemma}

\begin{proof}
	The assumption $\|g\|_0>d/2$ implies that the vector $w_k:=(g_i\bar g_{i+k})_i$ is not the zero vector for all $k$. Hence for each $k$, $a_{k,\ell}\neq 0$ for at least one $\ell$. Suppose by contradiction that $\rank(G(g))<2d$. Therefore, we must have some value of $k$ for which $a_{k,\ell}\neq 0$ for exactly one $\ell$. In particular $w_k$ is proportional to $(1,\omega^\ell,\omega^{2\ell},\ldots,\omega^{(d-1)\ell})$. There are two cases to consider:
	
	\begin{itemize}
		\item[Case I:] $k=0$. Then $w_0=(1,1,\ldots,1)/\sqrt{d}$. There must be some $k'\neq 0$ for which $a_{k',\ell}\neq 0$ for at most two values of $\ell$. Thus $w_{k'}$ equals to a vector $(a\omega^{im}+b\omega^{in})_i$, for some constants $a,b$. But as $|(w_{k'})_i|=1/\sqrt{d}$ for all $i$, $|a\omega^{im}+b\omega^{in}|$ is independent of $i$, which implies $a=0$ or $b=0$. WLOG $b=0$. Thus $g_i/ g_{i+k'}=da\omega^{im}$. Modifying $g$ by multiplicative transformation, and using $gcd(k',d)=1$, we may assume that  $k'=1$, which implies that $g_i=g_0(da)^i\omega^{mi(i-1)/2}$. Substituting $i=d$ we obtain $(da)^d=1$, hence $da=\omega^r$ for some $r$. Thus $g$ is a quadratic transformation of $(1,1,\ldots,1)/\sqrt{d}$ and $\rank(G(g))=d$. A contradiction.
		\item[Case II:] $k\neq 0$. Since $gcd(k,d)=1$, then by modifying $g$ be a multiplicative transformation we may assume that $k=1$, and $g_i\bar g_{i+1}=c\omega^{\ell i}$.
		By a phase quadratic transformation, we may reduce to the case $\ell=0$, so $g_i\bar g_{i+1}=c$ for all $i$. Dividing this by $\bar g_{i+1}g_{i+2}=\bar c$ we obtain $g_i/g_{i+2}=c/\bar c$. This implies that $|g_i|=|g_{i+2}|$ for all $i$ and as $d$ is odd, $|g_i|=1/\sqrt{d}$ for all $i$ and from this point the proof is identical to case I.
		 \end{itemize}
\end{proof}

\begin{remark}
It follows that for $d=3$ there is no $g$ with $\rank(G(g))=5$. If this was not the case, the generator  $g$  would satisfy $\|g\|_0=1$, but this in turn implies $\rank(G(g))=3$, a contradiction.
\end{remark}
We need introduce a definition and a preliminary result that is interesting in its own right.

\begin{defn}
    The \emph{density} $\delta=\delta(f)$ of a polynomial $f(X)\in F[X]$ over a field $F$ is the number of nonzero coefficients in $f$. 
\end{defn}
We denote $\mu_d$ the group of complex roots of unity of order $d$. The main part of the proof lies in the following result.   
\begin{proposition}
    Let $d$ be a prime integer and $f(X)\in \C[X]$ a polynomial of degree less than $d$. Then the number of roots of $f$ which are in $\mu_d$ is at most $\delta(f)-1$.
\end{proposition}

\begin{proof}
    The result is an immediate consequence of the fact that all minors of $d\times d$  DFT matrix are nonzero when $d$ is prime, see \cite[Theorem 6]{10.2307/2041358} or \cite[Theorem 4]{DELVAUX20081587}. Nonetheless, we provide here an algebraic number field argument.

    Suppose, by contradiction, that the polynomial $f$ has at least $\delta(f)$ roots in $\mu_d$. Pick $\delta(f)$ roots, $\alpha_i=\omega^{e_i}$, $0\le i< \delta(f)$. Let $h=\sum_{j=0}^{\delta(h)-1} c_jX^{h_j}$ be the polynomial of degree less than $d$, with the smallest possible density  ($\le \delta(f)$) having all the $\alpha_i$ as roots.

    Then the coefficients of $h$ are solution to a linear system of equations given by the vectors $v_j:=(\omega^{e_ih_j})_i\in \mu_d^{\delta(f)}$, $0\le j\le \delta(h)-1$. By the assumption on $h$, this is the minimal linear dependency. The rank of the matrix $V=(\omega^{e_ih_j})_{i,j}$ is $\delta(h)-1$, thus there are $\delta(h)-1$ independent rows, and the linear dependency coefficients can be read from the $\delta(h)-1$ size minors belonging to the submatrix corresponding to these rows. In particular the coefficients belong to the cyclotomic field $K:=\mathbb Q(\omega)$ and $h$ is proportional to a polynomial in $K[X]$. WLOG we assume that $h\in K[X]$.\\
    
    The maximal order of $K$ is known to be the ring $\Z[\omega]$, and by again rescaling $h$ we may assume that $h\in \Z[\omega][X]$. There is a ring homomorphism $\phi:\Z[\omega]\to \Z_d$. The kernel of $\phi$ is the unique prime ideal $\mathfrak D$ above $d$, which is known to be principal, and generated by $\omega-1$. We extend $\phi$ to a ring homomorphism (denoted again by $\phi$) $\phi:\Z[\omega][X] \to \Z_d[X]$. Let $v_\mathfrak D(z)$ denote the $\mathfrak D$-adic valuation of $z\in K$, and let $v_\mathfrak D(h)=\min_i v_\mathfrak D(c_i)$. Then replacing $h$ by $h/(\omega-1)^{v_\mathfrak D(h)}$, we still have $h\in \Z[\omega][X]$, and at least on $c_i\notin \mathfrak D$. In particular $\phi(h)\neq 0$.\\
    
    Now, $h(X)=h_0(X)\prod_i(X-\alpha_i)$, and still $h_0(X)\in \Z[\omega][X]$. This can be proved by induction by dividing $g(X)$ successively by each $(X-\alpha_i)$. For example, in the first step write $h(X)=h(X-\alpha_0+\alpha_0)$ and expand each monomial around $X-\alpha_0$ using the binomial formula.
    
    On applying $\phi$ we obtain
    $$\phi(h)(X)=(X-1)^{\delta(f)}\phi(h_0)(X).$$ In particular, all first $\delta(h)\le \delta(f)$ derivatives of $\phi(h)$ vanish at $X=1$:
    $$ \phi(h)^{(k)}(X=1)=\sum_i c_i h_i(h_i-1)\cdots (h_i-k+1)=0 \mod d, \ \ \forall k<\delta(h).$$
    Hence, since not all $c_i$ are $0\mod d$, the matrix $W=\big(h_i(h_i-1)\cdots (h_i-k+1)\big)_{i,k} \in (\Z_d)^{\delta(h),\delta(h)}$ has linearly dependent rows. Notice that the $k$th column of $W$ are the values of polynomial $F_k(X)=X(X-1)\cdots(X-k+1)$ substituted at $X=h_i$. Thus by performing column elementary operations on $W$, we may clear the lower terms in $F_k(X)$, and our matrix is Gauss equivalent to the matrix $U=(h_i^k)_{i,k}$. But $0\le h_i<d$ are distinct and $U$ is the Vandermonde matrix in the field $\Z_d$, hence $\det(W)=\det(U)\neq 0$ in $\Z_d$. This is a contradiction, and the theorem is proved.
\end{proof}

We are now ready to prove:

\begin{theorem}
    For an odd prime $d$, there is no unit vector $g$ with $d<\rank(G(g))<2d$.
\end{theorem}
\begin{proof}
   
    Suppose that there is such $g$.
    We know by Lemma \ref{lem:supp} that $\|g\|_0<d/2$. Then for each $k$,  $\|w_k\|_0<d/2$. Suppose that $k$ is taken such that $w_k\neq 0$. By the above theorem, the DFT $\widehat{w_k}$ can have  at most $(d-3)/2$ zero entries, hence the number of $\ell$ such that $a_{k,\ell}(g)\neq 0$ is at least $(d+3)/2$. This implies that there can be at most $3$ values of $k$ such that $w_k\neq 0$. This implies in turn that $\|g\|_0\le 2$. But in this case the conclusion of the theorem follows from Proposition \ref{Prop:proertiesrankG}.  
\end{proof}


We conclude the paper by proving the rank of Gabor POVMs in dimensions $4$ and $5$ when the generator $g$ does not have full support. 


\begin{proposition}\label{example:dim45}
Suppose that $g$ is a unit-norm vector. The following statements hold.
\begin{enumerate}
    \item[(i)] If $g\in \C^4$, then 
\[ \rank(G(g))=\begin{cases} 
      4 &   \text{if} \, \, \,  \|g\|_0=1 \\
      4,6,8,11,\, \text{or}\, 12 &  \text{if} \, \, \, \|g\|_0=2\\
      11, 12, 13, 14, 15, \, \text{or}\,   16 &  \text{if} \, \, \, \|g\|_0=3
   \end{cases}
\]
\item[(ii)] If $g\in \C^5$, then 
    \[\rank(G(g))=\begin{cases}
        5 &  \text{if} \, \, \,  \|g\|_0=1\\
        15 &  \text{if} \, \, \,  \|g\|_0=2\\
        21, 23, \, \text{or}\, 25 &  \text{if} \, \, \,  \|g\|_0=3
    \end{cases}
    \]

\end{enumerate}
\end{proposition}

\begin{proof}
The result for $\|g\|_0=1,2$ can be obtained form Theorem~\ref{thm:rankd} and Proposition \ref{Prop:proertiesrankG}.
\begin{enumerate}
    \item[(i)]








Assume now  $\|g\|_0=3$. Without loss of generality, let $g_0,g_1,g_2\neq 0$ and $g_3=0$. We have $\|\hat{w}_0\|_0=$3 or 4, $\|\hat{w}_1\|_0=\|\hat{w}_3\|_0=$3 or 4, and $\|\hat{w}_2\|_0=$2 or 4.
\begin{itemize}
    \item $\|\hat{w}_0\|_0=3$ if $|g_0|^2+|g_2|^2\neq |g_1|^2$, else $\|\hat{w}_0\|_0=4$.
   
    \item $\|\hat{w}_2\|_0=2$ if $g_0\bar{g}_2=\pm g_2\bar{g}_0$, else $\|\hat{w}_2\|_0=4$.
    \item $\|\hat{w}_1\|_0=\|\hat{w}_3\|_0=3$ if $g_0\bar{g}_1+\omega^{k}g_1\bar{g}_2=0$ for some $k\in \Z_4$, else  $\|\hat{w}_1\|_0=\|\hat{w}_3\|_0=4$. 
\end{itemize}
Since all possible combinations of ($||\hat w_0||_0,||\hat w_1||_0,\|\hat w_2\|_0)$ can be obtained, we can conclude that $\rank(G(g))$ can be any integer between 11 and 16. 
\item[(ii)] Next, suppose $\|g\|_0=3$. We assume $g_0,g_1,g_2\neq 0$. All other possibilities can be obtained from additive and multiplicative translation from this vector. Then
$$ \|w_0\|_0=3, \, \,\|w_1\|_0=\|w_4\|_0=2, \, \, \text{and}\, \, \|w_2\|_0=\|w_3\|_0=1.$$
So $\|\hat{w}_0\|_0=3$ or 5; $\|\hat{w}_2\|_0=\|\hat{w}_3\|_0=5$; and $\|\hat{w}_1\|_0=\|\hat{w}_4\|_0=$4 or 5.
\begin{itemize}
\item $\|\hat{w}_0\|_0=3$ if $w_0=(1, -2\cos(4\pi/5),1,0,0)$. Then $g$ is equivalent to \\$(1,\sqrt{-2\cos(4\pi/5)})e^{i\theta_1},e^{i\theta_2},0,0)$ for $\theta_1,\theta_2\in [0,2\pi]$. Otherwise $\|\hat{w}_0\|_0=5$.
\item $\|\hat{w}_1\|_0=\|\hat{w}_4\|_0=4$ if and only if $g$ is equivalent to a scalar multiple of $(1,g_1,\frac{-g_1}{\overline{g_1}}\omega^j,0,0)$ for some $j\in \Z_5$.  Otherwise $\|\hat{w}_1\|_0=\|\hat{w}_4\|_0=5$.  
\end{itemize}
Since all combinations of the pair ($||\hat w_0||_0,||\hat w_1||_0)$ can be obtained, we have $\rank(G(g))=$21, 23 or 25.
\end{enumerate}

\end{proof}

\section*{Acknowledgements}  S.~Kang and K.~A.~Okoudjou were partially supported by  the U. S.\ Army Research Office  grant  W911NF1610008,  the National Science Foundation grant DMS 1814253,  and an MLK  visiting professorship at MIT.

\bibliographystyle{amsplain}
\bibliography{bibli}
\end{document}